\documentclass[11pt,final]{article}
\usepackage[textwidth=5.5in,textheight=9.5in,centering]{geometry}%
\usepackage{titlesec,titletoc}
\usepackage{etex}
\usepackage{mathtools}
\usepackage{pdfpages}
\usepackage{url}
\usepackage{amsmath}
\usepackage{amsxtra,amscd}
\usepackage{amsfonts}
\usepackage{amssymb}
\usepackage{graphicx}
\usepackage[amsmath,hyperref,thmmarks]{ntheorem}
\usepackage{natbib}
\usepackage{verbatim}
\usepackage{mathptmx}
\usepackage{enumitem}
\usepackage[notcite,notref]{showkeys}
\usepackage{url}
\usepackage[colorlinks=true,bookmarksnumbered=true,pdfpagemode=None,extension=pdf]%
{hyperref}%
\providecommand{\U}[1]{\protect\rule{.1in}{.1in}}
\theoremnumbering{arabic}
\theoremheaderfont{\scshape}
\RequirePackage{latexsym}
\theorembodyfont{\slshape}
\theoremseparator{.}
\newtheorem{X}{X}[section]

\newtheorem{corollary}[X]{Corollary}

\newtheorem{lemma}[X]{Lemma}
\newtheorem*{*lemma}{Lemma}
\newtheorem{proposition}[X]{Proposition}
\newtheorem*{*proposition}{Proposition}

\newtheorem{theorem}[X]{Theorem}
\theorembodyfont{\upshape}
\newtheorem{aside}[X]{Aside}
\newtheorem{definition}[X]{Definition}

\newtheorem*{*definition}{Definition}

\newtheorem*{*example}{Example}

\newtheorem{plain}[X]{}

\theorembodyfont{\footnotesize}
\newtheorem{nt}[X]{Notes}
\theorembodyfont{\normalsize}
\theoremstyle{nonumberplain}
\theoremheaderfont{\sc}
\theorembodyfont{\normalfont}
\theoremsymbol{\ensuremath{_\Box}}
\RequirePackage{amssymb}
\newtheorem{proof}{Proof}
\qedsymbol{\ensuremath{_\Box}}
\theoremclass{LaTeX}
\makeindex

\newcommand{\da}{\tikz[baseline=-0.6ex]{\draw[densely dashed,->,>=angle 90] (0,0) -- (3ex,0); }}

\contentsmargin{2.0em}
\dottedcontents{section}[2.3em]{}{1.8em}{1pc}
\dottedcontents{subsection}[6.0em]{}{1.8em}{1pc}
\titleformat*{\subsection}{\large\scshape}
\titleformat*{\subsubsection}{\slshape}

\usepackage{titlesec}                                                                            
\titleformat*{\section}{\LARGE\bfseries}
\titleformat*{\subsection}{\Large\itshape}
\titleformat*{\subsubsection}{\scshape}
\titleformat*{\paragraph}{\itshape}

\usepackage{enumitem}
\setitemize[1]{label=$\diamond$\hspace{0.07in}}
\setlist{nolistsep}

\usepackage[nottoc]{tocbibind}

\usepackage{array}

\usepackage{microtype}

\usepackage{tikz}
\usetikzlibrary{matrix,arrows,positioning,decorations.pathmorphing}
\usepackage{tikz-cd}
 
\usepackage{wrapfig}

\usepackage{etex}

\usepackage{url}
\usepackage{verbatim}

\usepackage[notcite,notref]{showkeys}

\usepackage{mathtools}

\usepackage[overload]{textcase}

\newcommand{\eb}[1]{{\itshape\bfseries#1}}
\renewcommand{\emph}{\eb}

\usepackage{natbib}
\let\cite\citealt

\hypersetup{linkcolor=blue,anchorcolor=blue,citecolor=blue,filecolor=blue,urlcolor=blue}
\hypersetup{pdftitle={A Proof of the Barsotti-Chevalley Theorem on Algebraic Groups},pdfauthor={James S. Milne}}

\newcommand{\bcomment}{}
\newcommand{\bfootnotesize}{\begin{footnotesize}}\newcommand\efootnotesize{\end{footnotesize}}
\newcommand{\bquote}{\begin{quote}}\newcommand\equote{\end{quote}}
\newcommand{\bsmall}{\begin{small}}\newcommand\esmall{\end{small}}
\newcommand{\btable}{\begin{table}}\newcommand{\etable}{\end{table}}

\newcommand{\edocument}{\end{document}}

\newcommand{\mdskip}{\medskip}
\newcommand{\tableoc}{\tableofcontents}

\def\1{{1\mkern-7mu1}}

\DeclareMathOperator{\Aut}{Aut}

\DeclareMathOperator{\Gal}{Gal}
\DeclareMathOperator{\GL}{GL}

\DeclareMathOperator{\Ker}{Ker}

\DeclareMathOperator{\Spec}{Spec}

\begin{document}

\title{A Proof of the Barsotti-Chevalley Theorem on Algebraic Groups}
\author{James S. Milne}
\date{December 7, 2013}
\maketitle

\begin{abstract}
A fundamental theorem of Barsotti and Chevalley states that every smooth
connected algebraic group over a perfect field is an extension of an abelian
variety by a smooth affine algebraic group. In 1956 Rosenlicht gave a short
proof of the theorem. We explain his proof in the language of modern algebraic geometry.

\end{abstract}

\tableoc

\mdskip The theorem in question is the following:

\begin{quote}
Every connected group variety $G$ over a perfect field contains a unique
connected affine normal subgroup variety $N$ such that $G/N$ is an abelian variety.
\end{quote}

\noindent According to Rosenlicht (1956),\nocite{rosenlicht1956} this theorem
\textquotedblleft was announced by Chevalley in 1953, together with a proof
whose basic idea was to consider the natural homomorphism from a connected
algebraic group to its Albanese variety and then apply the basic properties of
Albanese and Picard varieties.\textquotedblright\ Chevalley didn't publish his
proof until 1960. In the meantime, Barsotti had independently published a
proof (\cite{barsotti1955i, barsotti1955}).\footnote{\nocite{borel2001}Borel
(2001, p.156) writes: \textquotedblleft In 1953, as pointed out [by
Rosenlicht], Chevalley showed that any irreducible algebraic group $G$
contains a biggest normal linear algebraic subgroup, which is the kernel of a
morphism of $G$ onto an abelian variety\ldots. He did not publish it at that
time, only later\ldots. The argument is in principle the one alluded to by
Rosenlicht. It uses a theory of the Albanese and of linear systems of divisors
in a non-projective situation. It may be that Chevalley waited until more
foundational material was available. He does indeed refer to later papers for
it.\textquotedblright\ In his review of Chevalley's paper, Barsotti points out
that Chevalley uses a result from an expos\'{e} of Grothendieck in the 1956/58
S\'{e}minaire Chevalley, which Barsotti traces back to his own paper.}

In the paper just cited, Rosenlicht gives a proof of the theorem that is both
simpler and more elementary than that of Chevalley.\footnote{Rosenlicht (1956)
says that Barsotti's papers \textquotedblleft seem to follow a similar
method\textquotedblright. Barsotti's papers are very difficult to read, which
perhaps helps to explain why they are so often ignored.} It deserves to be
better known. In this expository article, we explain it in the language of
modern algebraic geometry.

\subsubsection{Terminology}

We work over a fixed $k$. By an algebraic scheme, we mean a scheme of finite
type over $k$. An algebraic group scheme is a group in the category of
algebraic schemes over $k$. We abbreviate \textquotedblleft algebraic group
scheme\textquotedblright\ to \textquotedblleft algebraic
group\textquotedblright. An algebraic variety over $k$ is a
geometrically-reduced separated algebraic scheme over $k$. By a
\textquotedblleft group variety\textquotedblright\ we mean a group in the
category of \textit{connected} algebraic varieties. Thus the group varieties
are exactly the smooth connected algebraic groups. A \textquotedblleft
point\textquotedblright\ of an algebraic scheme means \textquotedblleft closed
point\textquotedblright. Finally, \textquotedblleft largest\textquotedblright%
\ means \textquotedblleft unique maximal\textquotedblright.

\section{Affine algebraic subgroups}

It is convenient to regard an (affine) algebraic group over $k$ as a functor
from $k$-algebras to groups whose underlying functor to sets is representable
by an (affine) algebraic scheme. An algebraic subgroup $H$ of an algebraic
group $G$ is a group subfunctor representable by an algebraic scheme. Our
finiteness assumption implies that $H$ is a closed subscheme of $G$ (hence
affine if $G$ is affine). An algebraic subgroup $H$ is normal if $H(R)$ is
normal in $G(R)$ for all $k$-algebras $R$.

A sequence of algebraic groups%
\begin{equation}
1\rightarrow N\rightarrow G\rightarrow Q\rightarrow1 \label{eq1}%
\end{equation}
is exact if and only if the sequence%
\[
1\rightarrow N(R)\rightarrow G(R)\rightarrow Q(R)
\]
is exact for all $k$-algebras $R$ and every element of $Q(R)$ lifts to
$G(R^{\prime})$ for some faithfully flat $R$-algebra $R^{\prime}$
(equivalently, $G\rightarrow Q$ is faithfully flat). Then $N$ is a normal
algebraic subgroup, and every normal algebraic subgroup arises in this way
from an essentially unique exact sequence (and we usually denote $Q$ by $G/N$).

The Noether isomorphism theorems hold for algebraic groups (and affine
algebraic groups) over a field. For example, if $N$ and $H$ are algebraic
subgroups of an algebraic group $G$ with $N$ normal, then there is an exact
sequence%
\[
1\rightarrow N\cap H\rightarrow N\rtimes_{\theta}H\rightarrow Q\rightarrow1.
\]
where $N\rtimes_{\theta}H$ is the semidirect product of $N$ and $H$ with
respect to the obvious action of $H$ on $N$. The quotient $Q$ can be
identified with the algebraic subgroup $NH$ of $G$ whose $R$-points are the
elements of $G(R)$ that lie in $N(R^{\prime})H(R^{\prime})$ for some
faithfully flat $R$-algebra $R^{\prime}$. The natural map of functors
$H\rightarrow NH/N$ determines an isomorphism%
\[
H/N\cap H\rightarrow NH/N
\]
of algebraic groups. For all of this, see my notes AGS (and the references
therein), or SGA 3, VI$_{A}$

\begin{lemma}
\label{b1}Let%
\[
1\rightarrow N\rightarrow G\rightarrow Q\rightarrow1
\]
be an exact sequence of algebraic groups.

\begin{enumerate}
\item If $N$ and $Q$ are affine (resp. smooth, resp. connected), then $G$ is
affine (resp. smooth, resp. connected).

\item If $G$ is affine (resp. smooth, resp. connected), then so also is $Q$.
\end{enumerate}
\end{lemma}

\begin{proof}
(a) Assume $N$ and $Q$ are affine. The morphism $G\rightarrow Q$ is faithfully
flat with affine fibres. Now $G\times_{Q}G\simeq G\times N$, and so the
morphism $G\times_{Q}G\rightarrow G$ is affine. By faithfully flat descent,
the morphism $G\rightarrow Q$ is affine. As $Q$ is affine, so also is $G$.

If $N$ and $Q$ are smooth, then $G\rightarrow Q$ has smooth fibres of constant
dimension, and so it is smooth. As $Q$ is smooth, so is $G$.

Let $\pi_{0}(G)$ be the group of connected components of $G$; it is an
\'{e}tale algebraic group, and the natural map $G\rightarrow\pi_{0}(G)$ is
universal among homomorphisms from $G$ to \'{e}tale algebraic groups. If $N$
is connected, then $G\twoheadrightarrow\pi_{0}(G)$ factors through $Q$, and
hence through $\pi_{0}(Q)$, which is trivial if $Q$ is also connected.

(b) That quotients of affine algebraic groups by normal algebraic subgroups
are affine is part of the general theory, discussed above.

To show that $Q$ is smooth, it suffices to show that it is geometrically
reduced but, because the map $G\rightarrow Q$ is faithfully flat, a nonzero
nilpotent local section of $\mathcal{O}_{Q_{\bar{k}}}$ will give a nonzero
nilpotent local section of $\mathcal{O}{}_{G_{\bar{k}}}$.

The faithfully flat homomorphism $G\rightarrow Q\rightarrow\pi_{0}(Q)$ factors
through $\pi_{0}(G)$, and so $\pi_{0}(Q)$ is trivial if $\pi_{0}(G)$ is.
\end{proof}

In particular, an extension of affine group varieties is again an affine group
variety, and a quotient of a group variety by a normal algebraic subgroup is a
group variety.

\begin{lemma}
\label{b2}Let $H$ and $N$ be algebraic subgroups of an algebraic group $G$
with $N$ normal. If $H$ and $N$ are affine (resp. connected, resp. smooth),
then $HN$ is affine (resp. connected, resp. smooth).
\end{lemma}

\begin{proof}
Consider the diagram%
\[
\begin{tikzcd}
1\arrow{r}&N\arrow{r}&HN\arrow{r}&HN/N\arrow{r}&1\\
&&&H/H\cap N.\arrow{u}[swap]{\simeq}
\end{tikzcd}
\]
If $H$ is affine (resp. connected, resp. smooth), then so also is the quotient
$H/H\cap N$ (by \ref{b1}b); hence $HN/N$ is affine (resp. connected, resp.
smooth), and it follows from (\ref{b1}a) that the same is true of $HN$.
\end{proof}

\begin{lemma}
\label{b3}Every algebraic group contains a largest smooth connected affine
normal algebraic subgroup (i.e., a largest affine normal subgroup variety).
\end{lemma}

\begin{proof}
Let $H$ and $N$ be maximal smooth connected affine normal algebraic subgroups
of an algebraic group. Then $HN$ also has these properties by (\ref{b2}a), and
so $H=HN=N$.
\end{proof}

\begin{definition}
\label{b4}A \emph{pseudo-abelian} variety is a group variety such that every
affine normal subgroup variety is trivial.
\end{definition}

\begin{proposition}
\label{b5}Every group variety $G$ can be written as an extension%
\[
1\rightarrow N\rightarrow G\rightarrow Q\rightarrow1
\]
of a pseudo-abelian variety $Q$ by a normal affine subgroup variety $N$ in
exactly one way.
\end{proposition}

\begin{proof}
Let $N$ be the largest normal affine subgroup variety of $G$ (see \ref{b3}),
and let $Q=G/N$. If $Q$ is not pseudo-abelian, then it contains a nontrivial
normal affine subgroup variety $H$. Let $H^{\prime}$ be the inverse image of
$H$ in $G$. From the exact sequence%
\[
1\rightarrow N\rightarrow H^{\prime}\rightarrow H\rightarrow1
\]
and (\ref{b1}) we see that $H^{\prime}$ is an affine subgroup variety of $G$.
Because $H$ is normal in $Q$, $H^{\prime}$ is normal in $G$, and so this
contradicts the definition of $N$. Hence $Q$ is a pseudo-abelian variety.

In order for $G/N$ to be pseudo-abelian, $N$ must be maximal among the normal
affine subgroup varieties of $G$; therefore it is unique (\ref{b3}).
\end{proof}

\begin{nt}
\label{b5a}Most of the generalities on algebraic group reviewed above are
proved in \S \S 2,3 of Rosenlicht's paper except that, as he worked in the
category of smooth algebraic groups, the isomorphism theorems hold only up to
a purely inseparable isogeny.
\end{nt}

\section{Rosenlicht's decomposition theorem}

Recall that an abelian variety is defined to be a complete group variety.
Abelian varieties are pseudo-abelian (every affine subgroup variety is closed;
hence complete; hence trivial). A rational map $\phi\colon X\da Y$ of
algebraic varieties is an equivalence class of pairs $(U,\phi_{U})$ with $U$ a
dense open subset of $X$ and $\phi_{U}$ a morphism $U\rightarrow Y$; in the
equivalence class, there is a pair with $U$ largest (and $U$ is called
\textquotedblleft the open subvariety on which $\phi$ is
defined.\textquotedblright) We shall need to use the following results, which
can be found, for example, in \cite{milne1986ab}.

\begin{plain}
\label{b6a} Every rational map from a normal variety to a complete variety is
defined on an open set whose complement has codimension $\geq2$ (ibid. 3.2).
\end{plain}

\begin{plain}
\label{b6b}A rational map from a smooth variety to a group variety is defined
on an open set whose complement is either empty or has pure codimension $1$
(ibid. 3.3).
\end{plain}

\begin{plain}
\label{b6c}Every rational map from a smooth variety $V$ to an abelian variety
$A$ is defined on the whole of $V$ (combine \ref{b6a} and \ref{b6b}).
\end{plain}

\begin{plain}
\label{b6d}Every morphism from a group variety to an abelian variety is the
composite of a homomorphism with a translation (ibid. 3.6).
\end{plain}

\begin{plain}
\label{b6e}Every abelian variety is commutative (apply (\ref{b6d}) to the map
$x\mapsto x^{-1}$).
\end{plain}

By an \emph{abelian subvariety} of an algebraic group we mean a complete
subgroup variety.

\begin{proposition}
\label{b6h}Let $G\times X\rightarrow X$ be an algebraic group acting
faithfully on an algebraic space $X$. If there is a fixed point $P$, then $G$
is affine.
\end{proposition}

\begin{proof}
Because $G$ fixes $P$, it acts on the local ring $\mathcal{O}{}_{P}$ at $P$.
The formation of $\mathcal{O}{}_{P}/\mathfrak{m}{}_{P}^{n+1}$ commutes with
extension of the base, and so the action of $G$ defines a homomorphism
$G(R)\rightarrow\Aut(R\otimes_{k}\mathcal{O}{}_{P}/\mathfrak{m}{}_{P}^{n+1})$
for all $k$-algebras $R$. This is natural in $R$, and so arises from a
(unique) homomorphism $\rho_{n}\colon G\rightarrow\GL_{\mathcal{O}{}%
_{e}/\mathfrak{m}{}_{P}^{n+1}}$ of algebraic groups. Let $Z_{n}=\Ker(\rho
_{n})$, and let $Z$ be the intersection of the descending sequence of
algebraic subgroups $\cdots\supset Z_{n}\supset Z_{n+1}\supset\cdots$. Because
$G$ is noetherian, there exists an $n_{0}$ such that $Z=Z_{n}$ for all $n\geq
n_{0}$.

Let $\mathcal{I}{}$ be the ideal in $\mathcal{O}{}_{G}$ corresponding to the
closed algebraic subgroup $Z$ in $G$. Then $\mathcal{I}\mathcal{O}{}%
_{P}\subset\mathfrak{m}{}_{P}^{n}$ for all $n\geq n_{0}$, and so
$\mathcal{I}{}{}\mathcal{O}{}_{e}\subset\bigcap\nolimits_{n}\mathfrak{m}{}%
_{e}^{n}=0$ (Krull intersection theorem). It follows that $Z$ contains an open
neighbourhood of $e$, and therefore is open in $G$. As it is closed and $G$ is
connected, it equals $G$. Therefore the representation of $G$ on
$\mathcal{O}{}_{e}/\mathfrak{m}{}_{e}^{n+1}$ is faithful if $n\geq n_{0}$,
which shows that $G$ is affine.
\end{proof}

\begin{corollary}
\label{b6}Let $G$ be a connected algebraic group, and let $\mathcal{O}{}_{e}$
be the local ring at the neutral element $e$. The action of $G$ on itself by
conjugation defines a representation of $G$ on $\mathcal{O}{}_{e}%
/\mathfrak{m}{}_{e}^{n+1}$. For all sufficiently large $n$, the kernel of this
representation is the centre of $G$.
\end{corollary}

\begin{proof}
Apply the above proof to the faithful action $G/Z\times G\rightarrow G$.
\end{proof}

\begin{proposition}
\label{b7}Let $G$ be a connected algebraic group. Every abelian subvariety $A$
of $G$ is contained in the centre of $G$; in particular, it is a normal
subgroup variety.
\end{proposition}

\begin{proof}
Because $A$ is complete, it is contained in the kernel of $\rho_{n}\colon
G\rightarrow\GL_{\mathcal{O}{}_{e}/\mathfrak{m}{}^{n+1}}$ for all $n.$
\end{proof}

\begin{lemma}
\label{b8a}Let $G$ be a commutative group variety over $k$, and let $V\times
G\rightarrow V$ be a $G$-torsor. There exists a morphism $\phi\colon
V\rightarrow G$ and an integer $n$ such that $\phi(v+g)=\phi(v)+ng$ for all
$v\in V$, $g\in G$.
\end{lemma}

\begin{proof}
Suppose first that $V(k)$ contains a point $P$. Then there is a unique
$G$-equivariant isomorphism%
\[
\phi\colon V\rightarrow G
\]
sending $P$ to the neutral element of $G$, namely, the map sending a point $v$
of $V$ to the unique point $(v-P)$ of $G$ such that $P+(v-P)=v$. In this case
we can take $n=1$.

In general, because $V$ is an algebraic variety, there exists a $P\in V$ whose
residue field $K\overset{\textup{{\tiny def}}}{=}\kappa(P)$ is a finite
\textit{separable} extension of $k$ (of degree $n$, say). Let $P_{1}%
,\ldots,P_{n}$ be the $k^{\mathrm{al}}$-points of $V$ lying over $P$, and let
$\tilde{K}$ denote the Galois closure (over $k$) of $K$ in $k^{\mathrm{al}}$.
Then the $P_{i}$ lie in $V(\tilde{K})$. Let $\Gamma=\Gal(\tilde{K}/k)$.

We have a morphism%
\[
v\mapsto\sum\nolimits_{i=1}^{n}(v-P_{i})\colon V_{\tilde{K}}\rightarrow
G_{\tilde{K}}%
\]
defined over $\tilde{K}$. It is $\Gamma$-equivariant, and so arises from a
morphism $\phi\colon V\rightarrow G$. For $g\in G$,%
\[
\phi(v+g)=\sum\nolimits_{i=1}^{n}(v+g-P_{i})=\sum\nolimits_{i=1}^{n}%
(v-P_{i})+g=\phi(v)+ng.
\]

\end{proof}

The next theorem says that every abelian subvariety of an algebraic group has
an almost-complement. It is a key ingredient in Rosenlicht's proof of the
Barsotti-Chevalley theorem.

\begin{theorem}
[Rosenlicht decomposition theorem]\label{b8}Let $A$ be an abelian subvariety
of a group variety $G$. There exists a normal algebraic subgroup $N$ of $G$
such that the map%
\begin{equation}
(a,n)\mapsto an\colon A\times N\rightarrow G \label{eq2}%
\end{equation}
is a faithfully flat homomorphism with finite kernel. When $k$ is perfect, $N$
can be chosen to be smooth.
\end{theorem}

\begin{proof}
Because $A$ is normal in $G$ (see \ref{b7}), there exists a faithfully flat
homomorphism $\pi\colon G\rightarrow Q$ with kernel $A$. Because $A$ is
smooth, the map $\pi$ has smooth fibres of constant dimension and so is
smooth. Let $V\rightarrow\Spec(K)$ be the generic fibre of $\pi$. Then $V$ is
an $A_{K}$-torsor over $K$. The morphism $\phi\colon V\rightarrow A_{K}$ over
$K$ given by the lemma, extends to a rational map $G\da Q\times A$ over $k$.
On projecting to $A$, we get a rational map $G\da A$. This extends to a
morphism (see \ref{b6c})
\[
\phi^{\prime}\colon G\rightarrow A
\]
satisfying%
\[
\phi^{\prime}(g+a)=\phi^{\prime}(g)+na
\]
(because it does so on an open set). After we apply a translation, this will
be a homomorphism whose restriction to $A$ is multiplication by $n$ (see
\ref{b6d}).

The kernel of $\phi^{\prime}\colon G\rightarrow A$ is a normal algebraic
subgroup $N$ of $G$. Because $A$ is contained in the centre of $G$ (see
\ref{b7}), the map (\ref{eq2}) is a homomorphism. It is faithfully flat
because the homomorphism $A\rightarrow G/N\simeq A$ is multiplication by $n$,
and its kernel is $N\cap A$, which is the finite group scheme $A_{n}$.

When $k$ is perfect, we can replace $N$ with $N_{\mathrm{red}}$, which is a
smooth algebraic subgroup of $N$.
\end{proof}

\begin{nt}
The statements (\ref{b8a}) and (\ref{b8}) correspond to Theorem 14 and its
Corollary (p.434) in Rosenlicht's paper.
\end{nt}

\section{The quasi-projectivity of algebraic groups}

\begin{theorem}
\label{b6f}Every abelian variety is projective.
\end{theorem}

\begin{proof}
We sketch the proof in \cite{weil1957}. By a standard argument, we may assume
that the base field $k$ is algebraically closed. In order to prove that a
variety is projective, it suffices to find a linear system $D$ of divisors
that separates points and tangent vectors, for then any basis of the vector
space%
\[
L(D)=\{f\mid(f)+D\geq0\}\cup\{0\}
\]
provides a projective embedding. For an abelian variety $A$, it is easy to
find a finite family $(Z_{i})_{1\leq i\leq n}$ of irreducible divisors such
that $\bigcap\nolimits_{i=1}^{n}Z_{i}=\{0\}$ and $\bigcap\nolimits_{i=1}%
^{n}T_{0}(Z_{i})=\{0\}$; here $T_{0}$ denotes the tangent space at $0$. For
the divisor $D=\sum_{i=1}^{n}Z_{i}$, one shows that $3D$ separates points and
tangent vectors on $A$, and hence defines a projective embedding of $A$. The
$3$ is needed for the following consequence of the theorem of the square: for
all points $a,b,c$ of $A$ such that $a+b+c=0$,%
\[
D_{a}+D_{b}+D_{c}\sim3D,
\]
where $D_{a}$ is the translation of $D$ by $a$. See \cite{milne1986ab}, 7.1,
for the details.
\end{proof}

\begin{theorem}
\label{b6g}Every homogeneous space for a group variety is quasi-projective.
\end{theorem}

\begin{proof}
We sketch the proof in \cite{chow1957}. In extending Weil's proof (see above)
to a homogeneous space, two problems present themselves. First, the vector
space $L(D)$ may be infinite dimensional. To circumvent this problem, Chow
proves the following result (now known as Chow's lemma):
\begin{quote} Let $V$ be an
algebraic variety. There exists an open subvariety $U$ of a projective variety
$V^{\prime}$ and a surjective birational regular map $U\rightarrow V$; when
$V$ is complete, $U=V^{\prime}$.
\end{quote}
Chow then uses linear systems on
$V^{\prime}$ rather than $V$. The second problem is the absence of a theorem
of the square for $V^{\prime}$. Chow circumvents this problem by making use of
the Picard variety of $V^{\prime}$.
\end{proof}

\section{Rosenlicht's dichotomy}

Throughout this section, $k$ is algebraically closed. The next result is the
second key ingredient in Rosenlicht's proof of the Barsotti-Chevalley theorem.

\begin{proposition}
\label{b9}Let $G$ be a group variety over an algebraically closed field $k$.
Either $G$ is complete or it contains an affine algebraic subgroup of
dimension $>0$.
\end{proposition}

In the presence of sufficiently strong resolution, the proof of (\ref{b9}) is
easy. Assume that $G$ is not complete, and let $X$ denote $G$ regarded as a
left homogeneous space for $G$. We may hope that (one day) canonical
resolution will give us a complete variety $\bar{X}$ containing $X$ as an open
subvariety and such that (a) $\bar{X}\smallsetminus X$ is strict divisor with
normal crossings, and (b) the action of $G$ on $X$ extends to an action of $G$
on $\bar{X}$ (cf. \cite{villamayor1992}, 7.6.3). The action of $G$ on $\bar
{X}$ preserves $E\overset{\textup{{\tiny def}}}{=}\bar{X}\smallsetminus X$.
Let $P\in E$, and let $H$ be the isotropy group at $P$. Then $H$ is an
algebraic subgroup of $G$ of dimension at least $\dim G-\dim E=1$. As it fixes
$P$ and acts faithfully on $X$, it is affine (\ref{b6h}). 

Lacking such a resolution theorem, we follow Rosenlicht (and \cite{brion2013}).

\begin{definition}
\label{b9b}A rational action of a group variety $G$ on a variety $X$ is a
rational map $G\times X\da X$ satisfying the following two conditions:

\begin{enumerate}
\item Let $e$ be the neutral element of $G$, and let $x$ be an element of $X$.
If $e\cdot x$ is defined, then $e\cdot x=x$.

\item Let $g,h\in G$, and let $x\in X$. If $h\cdot x$ and $g\cdot(h\cdot x)$
are defined, then so also is $gh\cdot x$, and it equals $g\cdot(h\cdot x)$.
\end{enumerate}
\end{definition}

\begin{lemma}
\label{b9a}Let $G$ be a group variety and let $X$ be a variety. A rational map
$\cdot\colon G\times X\da X$ is a rational action of $G$ on $X$ if there
exists a dense open subvariety $X_{0}$ of $X$ such that $\cdot$ restricts to a
regular action $G\times X_{0}\rightarrow X_{0}$ of $G$ on $X_{0}$.
\end{lemma}

\begin{proof}
Let $U$ be the (largest) open subvariety of $G\times X$ on which $\cdot$ is defined.

(a) Consider the map $\left(  \{e\}\times X\right)  \cap U\rightarrow X$,
$(e,x)\mapsto e\cdot x$. This agrees with the map $(e,x)\mapsto x$ on the
dense open subset $\{e\}\times X_{0}$ of $\left(  \{e\}\times X\right)  \cap
U$, and so agrees with it on the whole of $\left(  \{e\}\times X\right)  \cap
U$.

(b) Let $\Gamma\subset G\times G\times X\times X\times X\times X$ be the
closure of the graph $\Gamma_{\alpha}$ of the regular map%
\[
\alpha\colon G\times G\times X_{0}\rightarrow X_{0}\times X_{0}\times
X_{0},\quad(u,v,z)\mapsto(v\cdot z,u\cdot(v\cdot z),uv\cdot z).
\]
Let $\Delta(X_{0})$ denote the diagonal in $X_{0}\times X_{0}$. Then
$\Gamma_{\alpha}$ lies in $G\times G\times X_{0}\times X_{0}\times\Delta
(X_{0})$, and $\Gamma$ is the closure of $\Gamma_{\alpha}$ in $G\times G\times
X\times X\times\Delta(X)$.

Let $\Gamma^{\prime}\subset G\times G\times X\times X\times X$ be the closure
of the graph $\Gamma_{\alpha^{\prime}}$ of the regular map%
\[
\alpha^{\prime}\colon G\times G\times X_{0}\rightarrow X_{0}\times X_{0}%
,\quad(u,v,z)\mapsto(v\cdot z,u\cdot(v\cdot z)).
\]
The projection $G\times G\times X\times X\times\Delta(X)\rightarrow G\times
G\times X\times X\times X$ is an isomorphism, and maps $\Gamma_{\alpha}$
isomorphically onto $\Gamma_{\alpha^{\prime}}$. Therefore, it maps $\Gamma$
isomorphically onto $\Gamma^{\prime}$.

Let $(g,h,x)$ be as in (b) of (\ref{b9b}). As $\alpha^{\prime}$ is defined at
$(g,h,x)$, the projection $\Gamma^{\prime}\rightarrow G\times G\times X$ is an
isomorphism over a neighbourhood of $(g,h,x)$. Therefore, the projection
$\Gamma\rightarrow G\times G\times X$ is an isomorphism over a (the same)
neighbourhood of $(g,h,x)$. This implies that $gh\cdot x$ is defined (and
equals $g\cdot(h\cdot x)$).
\end{proof}

\begin{lemma}
\label{b9c}Let $\alpha\colon X\rightarrow Y$ be a dominant regular map of
irreducible varieties with $Y$ complete and let $D$ be prime divisor on $X$
such that $D\rightarrow Y$ is not dominant. Then there exists a complete
variety $Y^{\prime}$ and a birational regular map $\beta\colon Y^{\prime
}\rightarrow Y$ such that $\beta^{-1}\alpha(D)$ is a divisor on $Y^{\prime}$.
\end{lemma}

\begin{proof}
We regard $k(Y)$ as a subfield of $k(X)$. Let $\mathcal{O}{}_{D}$ be the local
ring of $D$ (ring of functions defined on some open subset of $D$). Then
$\mathcal{O}{}_{D}$ is a discrete valuation ring, and we let $v$ denote the
corresponding discrete valuation on $k(X)$. Let $w=v|k(Y)$. Then $w$ is a
discrete valuation on $k(Y)$, and it is nontrivial because $D\rightarrow Y$ is
not dominant. Let $\mathcal{O}{}_{w}$ be its valuation ring and $k(w)$ its
residue field. Then%
\[
\text{\textrm{tr deg}}_{k}k(w)=\text{\textrm{tr deg}}_{k}k(v)-\text{\textrm{tr
deg}}_{k(w)}k(v)\geq(\dim X-1)-(\dim X-\dim Y)=n-1
\]
where $n=\dim Y$. Let $f_{1},\ldots,f_{n-1}$ be elements of $\mathcal{O}{}%
_{w}$ whose images in $k(w)$ are algebraically independent of $k$, and let
$Y^{\prime}$ be the graph of the rational map%
\[
y\mapsto(1\colon f_{1}(y)\colon\cdots\colon f_{n-1}(y))\colon Y\da\mathbb{P}%
{}^{n-1}\text{.}%
\]
The projection maps from $Y\times\mathbb{P}{}^{n-1}$ give a birational regular
map $\beta\colon Y^{\prime}\rightarrow Y$ and a regular map $q\colon
Y^{\prime}\rightarrow\mathbb{P}{}^{n-1}$. The rational map $q\circ(\beta
^{-1}\alpha)\colon X\da\mathbb{P}^{n-1}$ restricts to the map $x\mapsto
(f_{1}(x)\colon\cdots\colon f_{n}(x))$ on $D$, which is dominant, and so
$\left(  \beta^{-1}\alpha\right)  (D)$ is a prime divisor of $Y^{\prime}$. See
\cite{brion2013}, 2.3.5, for more details.
\end{proof}

\subsubsection{Proof of Proposition \ref{b9}}

We assume that $G$ is not complete, and we use induction on $\dim(G)$ to show
that it contains an affine algebraic subgroup of dimension $>0$.

Let $X=G$ regarded as a left principal homogeneous space. According to
(\ref{b6g}), we may embed $X$ as an open subvariety of a complete algebraic
variety $\bar{X}$. If the boundary $\bar{X}\smallsetminus X$ has codimension
$\geq2$, we blow up so that it has pure codimension $1$. Finally, we replace
$\bar{X}$ with its normalization, which doesn't change $X$.

Because $G\times\bar{X}$ is normal and $\bar{X}$ is complete, the rational map
$\alpha\colon G\times\bar{X}\da\bar{X}$ is defined on an open subset $U$ of
$G\times\bar{X}$ whose complement has codimension $\geq2$ (see \ref{b6a}). Let
$E=\bar{X}\smallsetminus X$. As $E$ has pure codimension $1$ in $\bar{X}$, the
set $U\cap(G\times E)$ is open and dense in $G\times E$.

Note that if $g\in G$ and $x\in E$ are such that $g\cdot x$ is defined, then
$g\cdot x\in E$ --- otherwise $g\cdot x\in X$, and so $g^{-1}\cdot(g\cdot x)$
is defined; but then $e\cdot x$ is defined, and equals $g^{-1}\cdot(g\cdot x)$
(by \ref{b9a}); but $g^{-1}\cdot(g\cdot x)\in X$ and $e\cdot x=x\in E$, which
is a contradiction. Therefore $\alpha$ restricts to a rational action of $G$
on $E$,
\[
G\times E\da E.
\]

Let $E_{1}$ be an irreducible component of $E$. On applying (\ref{b9c}) to the
rational map $G\times X\da X$, we obtain a birational regular map $X^{\prime
}\rightarrow X$ such that the image  of $E_{1}$ under $G\times X\da X^{\prime
}$ is a divisor $D$. We normalize $X^{\prime}$ and use the birational
isomorphism between $X$ and $X^{\prime}$ to replace $X$ with $X^{\prime}$.
This allows us to assume that there is irreducible component $E_{1}^{\prime}$
of $E$ whose image $D$ under $G\times X\da X$ is divisor. For $(g,x)$ in an
open subset of $G\times D$, $g\cdot x$ is defined and $x=h\cdot y$ for some
$h\in G$, $y\in E_{1}^{\prime}$. Now $g\cdot x=g\cdot(h\cdot y)=(gh)\cdot y\in
D$. Therefore $\alpha$ restricts to a rational action of $G$ on $D$.

The rational map%
\[
(g,x)\mapsto(g^{-1},g\cdot x)\colon G\times D\rightarrow G\times D
\]
is birational (its square as a rational map is the identity map). Its image
contains an open subset of $G\times D$, and therefore intersects $U$. This
means that there exists a pair $(g_{0},P)\in G\times D$ such that
$(g_{0},P)\in U$ and $(g_{0}^{-1},g_{0}\cdot P)\in U$, i.e., such that
$g_{0}\cdot P$ and $g_{0}^{-1}\cdot(g_{0}\cdot P)$ are both defined. Therefore
$e\cdot P$ is defined and equals $P$ (\ref{b9a}).

Let $H^{\prime}=\{g\in V\mid gP=P\}$. Then $e\in H^{\prime}$ and the
multiplication map $G\times G\rightarrow G$ defines a rational map $H^{\prime
}\times H^{\prime}\da H^{\prime}$. Therefore, the closure $H$ of $H^{\prime}$
in $G$ is an algebraic subgroup of $G$. Its dimension is at least $\dim G-\dim
D=1$.

Suppose first that $H\neq G$. If $H$ is not complete, then (by induction) it
contains an affine algebraic subgroup of dimension $>0$. On the other hand, if
$H$ is complete, then there exists an almost-complement $N$ of $H$ in $G$ (see
\ref{b8}), which is not complete because $G$ isn't. By induction, $N$ contains
an affine algebraic subgroup of dimension $>0$.

It remains to treat the case $H=G$. As $G$ fixes $P$, it acts on
$\mathcal{O}{}_{U,P}$, and, as in the proof of \ref{b6h},  the homomorphism
$\rho\colon G\rightarrow\GL(\mathcal{O}{}_{P}/\mathfrak{m}{}_{P}^{n})$ is
faithful for some $n$, which implies that $G$ itself is affine.

\section{The Barsotti-Chevalley theorem}

\begin{theorem}
[\cite{barsotti1955}, \cite{chevalley1960}, Raynaud]\label{b10}Let $G$ be a
group variety over a field $k$. There exists a smallest connected affine
normal algebraic subgroup $N$ of $G$ such that $G/N$ is an abelian variety.
When $N$ is smooth, its formation commutes with extension of the base field.
When $k$ is perfect, $N$ is smooth$.$
\end{theorem}

\begin{proof}
Let $N_{1}$ and $N_{2}$ be connected affine normal algebraic subgroups of $G$
such that $G/N_{1}$ and $G/N_{2}$ are abelian varieties. There is a closed
immersion $G/N_{1}\cap N_{2}\hookrightarrow G/N_{1}\times G/N_{2}$, and so
$G/N_{1}\cap N_{2}$ is also complete (hence an abelian variety). This shows
that, if there exists a connected affine normal algebraic subgroup $N$ of $G$
such that $G/N$ is an abelian variety, then there exists a smallest such subgroup.

Let $N$ be as in the theorem, and let $k^{\prime}$ be a field containing $k$.
Then $N_{k^{\prime}}$ is a connected affine normal algebraic subgroup of
$G_{k^{\prime}}$, and $G_{k^{\prime}}/N_{k^{\prime}}\simeq(G/N)_{k^{\prime}}$
is an abelian variety. However, $N_{k^{\prime}}$ need no longer be the
smallest such subgroup because, for example, $(N_{k^{\prime}})_{\mathrm{red}}$
may have these properties. If $N$ is smooth, then every proper algebraic
subgroup has dimension less than that of $N$, and so this problem doesn't arise.

We now assume that $k$ is algebraically closed and prove by induction on the
dimension of $G$ that it contains an affine normal subgroup variety $N$ such
that $G/N$ is an abelian variety.

Let $Z$ be the centre of $G$. If $Z_{\mathrm{red}}=1$, then the representation
of $G$ on the $k$-vector space $\mathcal{O}{}_{e}/\mathfrak{m}{}_{e}^{n+1}$
has finite kernel for $n$ sufficiently large (see \ref{b6}), which implies
that $G$ itself is affine. Therefore, we may assume that $Z_{\mathrm{red}}%
\neq1$.

If $Z_{\mathrm{red}}$ is complete, then there exists an almost-complement $N$
to $Z_{\mathrm{red}}$ (see \ref{b8}), which we may assume to be smooth. As
$\dim N<\dim G$, there exists an affine normal subgroup variety $N_{1}$ of $N$
such that $N/N_{1}$ is an abelian variety. Then $N_{1}$ is normal in $G$, and
$G/N_{1}$ is an abelian variety because the isogeny $Z_{\mathrm{red}}\times
N\twoheadrightarrow G$ induces an isogeny $Z\times N/N_{1}\twoheadrightarrow
G/N_{1}$.

If $Z_{\mathrm{red}}$ is not complete, then it contains an affine subgroup
variety $N$ of dimension $>0$ (see \ref{b9}). Now $N$ is normal in $G$, and
(by induction) $G/N$ contains an affine normal subgroup variety $N_{1}$ such
that $(G/N)/N_{1}$ is an abelian variety. The inverse image $N_{1}^{\prime}$
of $N_{1}$ in $G$ is again an affine normal subgroup variety (apply \ref{b1}),
and $G/N_{1}^{\prime}\simeq$ $(G/N)/N_{1}$ is an abelian variety.

When $k$ is algebraically closed, we have shown that $G$ contains a smallest
affine normal subgroup variety $N$ such that $G/N$ is an abelian variety. If
$G$ is defined over a perfect subfield $k_{0}$ of $k$, then the uniqueness of
$N$ implies that it is stable under the action of $\Gal(k/k_{0})$, and is
therefore defined over $k_{0}$ (by Galois descent theory). This completes the
proof of the theorem when the base field is perfect.

Now suppose that $k$ is nonperfect. We know that for some finite purely
inseparable extension $k^{\prime}$ of $k$, $G^{\prime}\overset
{\textup{{\tiny def}}}{=}G_{k^{\prime}}$ contains a connected affine normal
algebraic subgroup $N^{\prime}$ such that $G^{\prime}/N^{\prime}$ is an
abelian variety. We may suppose that $k^{\prime p}\subset k$. Consider the
Frobenius map%
\[
F\colon G^{\prime}\rightarrow G^{\prime(p)}\overset{\textup{{\tiny def}}}%
{=}G^{\prime}\otimes_{k^{\prime}}k^{\prime(1/p)}.
\]
Let $N$ be the pull-back under $F$ of the algebraic subgroup $N^{\prime(p)}$
of $G^{\prime(p)}$. If $\mathcal{I}{}^{\prime}\subset\mathcal{O}{}_{G^{\prime
}}$ is the sheaf of ideals defining $N^{\prime}$, then the sheaf of ideals
$\mathcal{I}$ defining $N$ is generated by the $p$th powers of the local
sections of $\mathcal{I}^{\prime}$. As $k^{\prime p}\subset k$, we see that
$\mathcal{I}{}$ is generated by local sections of $\mathcal{O}{}_{G}$, and,
hence, that $N$ is defined over $k$. Now $N$ is connected, normal, and affine,
and $G/N$ is an abelian variety (because $N_{k^{\prime}}\supset N^{\prime}$
and so $(G/N)_{k^{\prime}}$ is a quotient of $G_{k^{\prime}}/N^{\prime}$).
\end{proof}

\begin{corollary}
\label{b11}Every pseudo-abelian variety over a perfect field is an abelian variety.
\end{corollary}

\begin{proof}
If $k$ is perfect and $G$ is pseudo-abelian, then the algebraic subgroup $N$
in the theorem is trivial.
\end{proof}

\begin{corollary}
\label{b12a}Every pseudo-abelian variety is commutative.
\end{corollary}

\begin{proof}
By definition, the commutator subgroup $G^{\prime}$ of an algebraic group $G$
is the smallest normal algebraic subgroup such that $G/G^{\prime}$ is
commutative. It is smooth (resp. connected) if $G$ is smooth (resp. connected).

Let $G$ be a pseudo-abelian variety, and let $N$ be as in the theorem. As
$G/N$ is commutative (\ref{b6c}), $G^{\prime}\subset N$. Therefore $G^{\prime
}$ is affine. As it is smooth, connected, and normal, it is trivial.
\end{proof}

\begin{aside}
\label{b13}It is possible to replace Chow's theorem (\ref{b6g}) in the proof
of (\ref{b9}), hence (\ref{b10}), with the Nagata embedding theorem
(\cite{nagata1962i}, \cite{deligne2010}), and then deduce from (\ref{b10})
that every algebraic group $G$ is quasi-projective: let $N$ be an affine
normal algebraic subgroup of $G$ such that $G/N$ is an abelian variety; then
the map $G\rightarrow G/N$ is affine (see the proof of \ref{b1}a) and $G/N$ is
projective (see \ref{b6f}); this implies that $G$ is quasi-projective.
\end{aside}

\begin{aside}
\label{b13a}Originally Barsotti and Chevalley proved their theorem over an
algebraically closed field. Because of the uniqueness, descent theory shows
that the statement in fact holds over perfect fields. Raynaud extended it to
arbitrary fields (but then the affine algebraic subgroup need no longer be
smooth). Cf. \cite{boschLR1990}, 9.2, Theorem 1, p.243.
\end{aside}

\begin{aside}
\label{b13b}Let $G$ and $N$ be as in the statement of the theorem. Clearly the
map $G\rightarrow G/N$ is universal among maps from $G$ to an abelian variety
preserving the neutral elements. Therefore $G/N$ is the Albanese variety of
$G$. In his proof of Theorem \ref{b10}, Chevalley assumed the existence of an
Albanese variety $A$ for $G$, and proved that the kernel of the map
$G\rightarrow A$ is affine.
\end{aside}

\begin{nt}
\label{b8b} The statements (\ref{b9}) and (\ref{b10}) correspond to Lemma 1,
p.437 and Theorem 16, p.439 in Rosenlicht's paper.
\end{nt}

\section{Complements}

\subsubsection{Pseudo-abelian varieties}

The defect of Raynaud's extension of the Barsotti-Chevalley theorem to
nonperfect fields is that the affine normal algebraic subgroup $N$ need not be
smooth. We saw in (\ref{b5}) that every group variety is an extension of a
pseudo-abelian variety by a smooth connected normal affine algebraic group (in
a unique way). This result is only useful if we know something about
pseudo-abelian varieties. We saw in (\ref{b12a}) that they are commutative,
and Totaro (2013)\nocite{totaro2013} shows that every pseudo-abelian variety
$G$ is an extension of a unipotent group variety $U$ by an abelian variety
$A$,%
\[
1\rightarrow A\rightarrow G\rightarrow U\rightarrow1,
\]
in a unique way.

\subsubsection{Anti-affine groups}

Let $G$ be an algebraic group over a field $k$. Then $\mathcal{O}{}(G)$ has
the structure of a Hopf algebra, and so it arises from an affine group scheme,
which we denote $G^{\mathrm{aff}}$. There is a natural homomorphism
$\phi\colon G\rightarrow G^{\mathrm{aff}}$. Cf. \cite{rosenlicht1956}, p.432.

\begin{proposition}
\label{b15}The affine group scheme $G^{\mathrm{aff}}$ is algebraic, and the
natural map $\phi\colon G\rightarrow G^{\mathrm{aff}}$ is universal for
homomorphisms from $G$ to affine algebraic groups; in particular, it is
faithfully flat. The kernel $N$ of $\phi$ has the property that $\mathcal{O}%
{}(N)=k$.
\end{proposition}

\begin{proof}
\cite{demazureG1970} III, \S 3, 8.2 (p. 357).
\end{proof}

An algebraic group $G$ over $k$ such that $\mathcal{O}{}(G)=k$ is said to be
\emph{anti-affine}. Thus every algebraic group is an extension of an affine
algebraic group by an anti-affine algebraic group%
\[
1\rightarrow G_{\mathrm{ant}}\rightarrow G\rightarrow G^{\mathrm{aff}%
}\rightarrow1,
\]
in a unique way. It is known that every anti-affine algebraic group is smooth,
connected, and commutative; in fact, every homomorphism from an anti-affine
group to a connected algebraic group factors through the centre of the group
(ibid., 8.3). Therefore, a group variety is an extension of an affine group
variety by a central anti-affine group variety.

Over a field of nonzero characteristic, every anti-affine algebraic group is
an extension of an abelian variety by a torus, i.e., it is a semi-abelian
variety. Over a field of characteristic zero it may also be an extension of a
semi-abelian variety by a vector group. Not every such extension is
anti-affine, but those that are have been classified. See \cite{brion2009,
salas2001, salas2009}.

\subsubsection{A decomposition theorem}

\begin{theorem}
\label{b14}Let $G$ be a connected algebraic group over field $k$. Let
$G_{\mathrm{aff}}$ be the smallest connected affine normal algebraic subgroup
of $G$ such that $G/G_{\mathrm{aff}}$ is an abelian variety (see \ref{b10}),
and let $G_{\mathrm{ant}}$ be the smallest normal algebraic subgroup such that
$G/G_{\mathrm{ant}}$ is affine (see \ref{b15}). Then the multiplication map on
$G$ defines an exact sequence%
\[
1\rightarrow G_{\mathrm{aff}}\cap G_{\mathrm{ant}}\rightarrow G_{\mathrm{aff}%
}\times G_{\mathrm{ant}}\rightarrow G\rightarrow1.
\]

\end{theorem}

\begin{proof}
Because $G_{\mathrm{ant}}$ is contained in the centre of $G$ (see above), the
map $G_{\mathrm{aff}}\times G_{\mathrm{ant}}\rightarrow G$ is a homomorphism
of algebraic groups. It is faithfully flat because the quotient of $G$ by its
image is both affine and complete.
\end{proof}

\begin{nt}
Brion 2008, p.945, notes that (\ref{b14}) is a variant of the structure
theorems at the end of Rosenlicht's paper.
\end{nt}

\subsection{Acknowedgment}

I thank M. Brion for his comments on the first version of the article; in
particular, for pointing out a gap in the proof of the original Lemma 3.1, and
for alerting me to his forthcoming book, which gives a much more expansive
treatment of the questions examined in this note.

\bibliographystyle{cbe}
\bibliography{D:/Current/refs}

\bigskip James S. Milne,

Mathematics Department, University of Michigan, Ann Arbor, MI 48109, USA,

Email: jmilne@umich.edu

Webpage: \url{www.jmilne.org/math/}
\end{document}